\documentclass{article}

\usepackage{epsfig}
\usepackage{graphicx}
\usepackage{amsbsy}
\usepackage{amsmath}
\usepackage{amsfonts}
\usepackage{amssymb}
\usepackage{textcomp}
\usepackage{hyperref}
\usepackage{aliascnt}
\usepackage{xcolor}
\usepackage{enumerate}
\usepackage{accents}
\usepackage{mathtools}
\newcommand{\mcm}[3]{\newcommand{#1}[#2]{{\ensuremath{#3}}}} 

\mcm{\tuple}{1}{\langle #1 \rangle}
\mcm{\name}{1}{\ulcorner #1 \urcorner}
\mcm{\Nbb}{0}{\mathbb{N}}
\mcm{\Zbb}{0}{\mathbb{Z}}
\mcm{\Rbb}{0}{\mathbb{R}}
\mcm{\Cbb}{0}{\mathbb{C}}
\mcm{\Qbb}{0}{\mathbb{Q}}
\mcm{\Acal}{0}{\cal A}
\mcm{\Bcal}{0}{\cal B}
\mcm{\Ccal}{0}{\cal C}
\mcm{\Dcal}{0}{\cal D}
\mcm{\Ecal}{0}{\cal E}
\mcm{\Fcal}{0}{\cal F}
\mcm{\Gcal}{0}{\cal G}
\mcm{\Hcal}{0}{\cal H}
\mcm{\Ical}{0}{\cal I}
\mcm{\Jcal}{0}{\cal J}
\mcm{\Kcal}{0}{\cal K}
\mcm{\Lcal}{0}{\cal L}
\mcm{\Mcal}{0}{\cal M}
\mcm{\Ncal}{0}{\cal N}
\mcm{\Ocal}{0}{{\cal O}}
\mcm{\Pcal}{0}{{\cal P}}
\mcm{\Qcal}{0}{{\cal Q}}
\mcm{\Rcal}{0}{{\cal R}}
\mcm{\Scal}{0}{{\cal S}}
\mcm{\Tcal}{0}{{\cal T}}
\mcm{\Ucal}{0}{{\cal U}}
\mcm{\Vcal}{0}{{\cal V}}
\mcm{\Wcal}{0}{{\cal W}}
\mcm{\Xcal}{0}{{\cal X}}
\mcm{\Ycal}{0}{{\cal Y}}
\mcm{\Zcal}{0}{{\cal Z}}
\mcm{\Mfrak}{0}{\mathfrak M}

\mcm{\restric}{0}{\upharpoonright}
\mcm{\upset}{0}{\uparrow}
\mcm{\onto}{0}{\twoheadrightarrow}
\mcm{\smallNbb}{0}{{\small \mathbb{N}}}
\DeclareMathOperator{\preop}{op}
\mcm{\op}{0}{^{\preop}}

{\begin{array}{c}
\setlength{\unitlength}{1em}}%
{\end{array}}

\usepackage{amsthm}

\newcommand{\theoremize}[2]{\newaliascnt{#1}{thm} \newtheorem{#1}[#1]{#2} \aliascntresetthe{#1}}

\theoremstyle{plain}
\newtheorem{thm}{Theorem}[section]
\theoremize{lem}{Lemma}
\theoremize{skolem}{Skolem}
\theoremize{fact}{Fact}
\theoremize{sublem}{Sublemma}
\theoremize{claim}{Claim}
\theoremize{obs}{Observation}
\theoremize{prop}{Proposition}
\theoremize{cor}{Corollary}
\theoremize{que}{Question}
\theoremize{oque}{Open Question}
\theoremize{con}{Conjecture}

\theoremstyle{definition}
\theoremize{defn}{Definition}
\theoremize{rem}{Remark}
\theoremize{eg}{Example}
\theoremize{exercise}{Exercise}
\theoremstyle{plain}

\begin{document}
\title{Outerspatial 2-complexes:  
\newline
Extending the class of outerplanar graphs to three dimensions}
\author{Johannes Carmesin, Tsvetomir Mihaylov\\ University of Birmingham}

\maketitle

\begin{abstract}
We introduce the class of outerspatial $2$-complexes as the natural generalisation of the class of outerplanar graphs to three dimensions. 
Answering a question of O-joung Kwon, 
we prove that a locally $2$-connected $2$-complex is outerspatial if and only if it does not contain a surface of positive genus as a subcomplex and does not have a space minor that is a generalised cone over $K_4$ or $K_{2,3}$.

This is applied to nested plane embeddings of graphs; that is, plane embeddings constrained by conditions placed on a set of cycles of the graph.

\end{abstract}

\section{Introduction}

An important class of planar graphs is the class of \emph{outerplanar graphs}, those graphs with plane embeddings having a face containing all vertices. Kuratowski's\footnote{Wagner's characterisation is stated in terms of minors while Kuratowski characterisation is stated in terms of subdivisions. They are equivalent, but since Kuratowski result comes first, we refer to either statement as Kuratowski's theorem.} characterisation of planar graphs in terms of excluded minors implies a characterisation of outerplanar graphs in terms of excluded minors. 

This paper is part of a project aiming to extend theorems from planar graph theory to three dimensions. The starting point of this project is \cite{JC1}. In there a three-dimensional analogue of Kuratowski's theorem was proved: embeddability of simply connected 2-complexes in the 3-sphere was characterised by excluded \lq space minors\rq. O-joung Kwon asked\footnote{In private communication.} whether a similar result is true for a natural higher dimensional analogue of outerplanar graphs. 

One of the equivalent definitions for a graph to be outerplanar is that the $1$-dimensional cone over it is planar. We will take this definition one dimension higher and will call a $2$-complex \emph{outerspatial} if the $2$-dimensional cone over it is embeddable in $\mathbb{R}^3$.
Hence outerspatial $2$-complexes are the natural generalisation of outerplanar graphs to three dimensions. The main result of this paper answers the above mentioned question of O-joung Kwon affirmatively, and is the following.

\begin{thm} \label{mainintro}
A locally $2$-connected\footnote{For a definition of locally $2$-connected look at \autoref{l2c}} simple $2$-complex~$C$ is outerspatial if and only if it does not contain a surface of positive genus or a space minor with a link graph that is not outerplanar. 
\end{thm}

The obstructions given in the theorem above are necessary: firstly, the torus has genus one and any $2$-complex homeomorphic to a torus is not outerspatial by  \autoref{a-sphere-cone} below. Secondly, cones over non-outerplanar graphs have link graphs that are not outerplanar and thus are not outerspatial $2$-complexes by \autoref{link_outerplanar} below.

Using the three-dimensional Kuratowski characterisation \cite[Theorem 1.3]{JC1} one can obtain a characterisation of the class of outerspatial $2$-complexes in terms of excluded minors. However, the set in question cannot be defined clearly and is too large to be of practical use. Our main result, \autoref{mainintro}, provides a simple and short forbidden structures characterisation.

In order to have such a short list as in \autoref{mainintro}, the assumption of local $2$-connectedness is also necessary. Indeed, triangulations of the torus with a single disc removed (and various slight modifications of higher genus surfaces) are excluded minors of the class of outerspatial $2$-complexes. Hence for this super-class the list of excluded minors is much more complicated than that of \autoref{mainintro}. 

\begin{eg}
Consider a $2$-complex built by gluing a set of triangulated $2$-spheres step by step such that at each step the gluing set is a single face. Any $2$-complex built this way is an example of a locally $2$-connected outerspatial $2$-complex. We will show in \autoref{sec5} that all such $2$-complexes can be constructed in this way.
\end{eg}
\begin{eg} \label{bing}
The topological space Bing House, as described in \cite[Chapter 0]{hatcher}, has only one chamber but it is not outerspatial. It is not outerspatial because it has a minor which has a link graph $K_4$.
\end{eg}

\begin{rem}\label{rem99}
Throughout this paper, faces of $2$-complexes are bounded by genuine cycles (of arbitrary length), as restricting to faces of size three makes the question one-dimensional. See \autoref{main_naive_lemma} for details. 
\end{rem}

In this paper we find a correspondence between particular plane embeddings and outerspatial embeddability in $3$-space, as follows. 
Given a set of cycles $\Ccal$ in a plane graph $G$, we say that the pair $(G,\Ccal)$ has a \emph{nested plane embedding} if $G$ has an embedding in the plane such that any two cycles in $\Ccal$ do not intersect internally (for a more precise definition, look at \autoref{nested_emb}). 
The $2$-complex \emph{associated to $(G,\Ccal)$} is the $2$-complex whose 1-skeleton is $G$ and whose set of faces is $\Ccal$. We prove the following connection between outerspatial $2$-complexes and nested plane embeddings. 

\begin{cor} 
A graph $G$ together with a set of cycles $\mathcal{C}$ has a nested plane embedding if and only if its associated $2$-complex is outerspatial.
\end{cor}

This result follows from \autoref{plane_to_outerspatial}. Given this corollary, \autoref{mainintro} can be applied directly to characterise the existence of nested plane embeddings of graphs.

{\bf Related Results.} The two most important concepts of this paper are embeddings of $2$-complexes in $3$-space and nested plane embeddings of graphs. Our methods for embedding $2$-complexes is related to and based on the series of papers on this topic \cite{JC1,JC2,JC3,JC4,JC5}. Some previous works related to nested plane embeddings focus on the triangle case. Such special types of nested plane embeddings are studied in papers \cite{nested_h,min_1,min_2}. The first one explores properties of graphs in relation to structural information on these nested triangles. The other two papers use nestedness as a tool to find an example of a minimal area straight line drawings of planar graphs. The term `laminar' is a general notion relating to sets, but is also used with the same meaning as our definition in terms of nested plane embeddings. Its usage in the context of cycles is motivated by the fact that the interiors of the faces bounded by a set of laminar cycles form a family of laminar subsets of $\mathbb{R}^2$. Laminar cycles play central part in the papers \cite{lamgoe,lamfio,lamasad,lamepp}. In \cite{lamgoe} the main problem is finding a minimum-weight set of vertices that meets all cycles in the subset. There the authors optimise an algorithm that they have found over laminar sets of cycles.  In \cite{lamfio} the aim is to bound the number of odd cycle vertex packings by the number of odd cycle vertex transversal. A main idea in proving this is considering laminar sets of odd cycles. In \cite{lamasad} laminar cycles are used to count $3$-colourings of triangle-free planar graphs. In \cite{lamepp} laminar cycles are used to find maximal sets of laminar $3$-separators in $3$-connected planar graphs. In \cite{induced_packing} it was shown that if a set of nested cycles in a graph satisfies some further properties, then this graph has a packing of $k$ odd cycles if and only if $G-v$ does for some specific $v$. There is also the notion of simply nested $k$-outerplanar graphs, which is somewhat related to our project; see \cite{simply_nested} for definitions.

The structure of this paper is as follows.  In the second section we give some basic definitions and prove some initial results. In the third section we build up to and state the Core Lemma -- \autoref{core_lemma}, which is the key component of the proof of \autoref{mainintro}. In the fourth section we prove the main techniques needed for the Core Lemma and we complete its proof and consequently prove \autoref{mainintro}. The fifth section is devoted to deriving some properties of locally $2$-connected simple outerspatial $2$-complexes following from our results.

\section{Basic definitions and initial approaches} \label{section_defns}

We start this section by giving basic definitions related to $2$-complexes. Next, we will explore various ways of defining the concept of outerspatiality and provide a brief explanation as to why we believe our definition to be the most effective. Then we start building the theory needed to prove the main theorem. At the end we show a proposition which on its own proves a rudimentary version of the main result, but the idea behind it is also quite useful later on.

Let us note that in this paper when we talk about a graph, it is assumed that it can have parallel edges and loops. We will now define what a $2$-complex is.
\begin{defn} \label{complex}
A \emph{$2$-complex} is a graph $G=(V,E)$ together with a set $F$ of cycles, called its \emph{faces}.

\end{defn}
\begin{rem}
    In this paper, we will assume that all edges of a $2$-complex~$C$ lie on some face of~$C$.
\end{rem}
\begin{defn}
We will call a $2$-complex \emph{simple} if it (that is, its underlying graph) does not have loops or parallel edges.\footnote{Parallel faces are not relevant to the question of embeddability. That is why we do not mention them in the definition of \emph{simple} 2-complexes.}
\end{defn}

\begin{eg}
A $2$-dimensional simplicial complex is a simple $2$-complex where all faces have three edges.
\end{eg}
\begin{defn}
The $1$-skeleton of a complex $C = (V,E,F)$ is the graph $G=(V,E)$.
\end{defn}

A notion that underlies this paper is that of space minors, the $3$-dimensional analogue of graph minors.

\begin{defn} \label{minoring}
A \emph{space minor\footnote{The systematic study of the minor relation was initiated by Wagner.}} of a $2$-complex is obtained by successively performing one of these two operations.
\begin{enumerate}
\item contracting an edge that is not a loop;
\item deleting a face (and all edges or vertices only incident with that face);
\end{enumerate}
\end{defn}

\begin{rem} \label{pres}
For detailed discussion on these operations and a proof that they are well-founded and preserve embeddability in 3-space, see \cite{JC1}.
\end{rem}



The aim of this paper is to extend the notion of outerplanarity of graphs to three dimensions. Before we do this, we need a definition of outerplanarity that translates well to $2$-complexes. There are two ways of defining outerplanar graphs that suit our purposes. One is to find an `outer' face containing all vertices and the other is through planarity of the cone. These two definitions are shown below.
\begin{defn}
Let $G$ be a graph. Take the disjoint union of $G$ and an additional vertex $t$ and connect $t$ to all vertices of $G$ by an edge. The resulting graph is called the \emph{(1-dimensional) cone} over $G$ and the vertex $t$ is called the top of the cone.
\end{defn} 
\begin{defn}\label{OP1}
(Outerplanarity criterion 1)
A graph $G$ is outerplanar if it can be embedded in the plane in such a way that there is a face of the embedding containing all vertices of $G$.
\end{defn}

\begin{defn}\label{OP2}
(Outerplanarity criterion 2)
A graph $G$ is outerplanar if the $1$-dimensional cone over $G$ is planar.
\end{defn}

The fact that these two definitions are equivalent is a well-known result. A proof sketch goes as follows. 

If a graph $G$ is outerplanar by criterion $1$, we can embed it in the plane so that there exists a face containing all vertices of~$G$. Then we can add a vertex on the interior of this face and connect it to all vertices of this face and hence all vertices of the graph. Thus, we embedded the cone over $G$ in the plane, which shows that criterion $1$ implies criterion $2$.

If $G$ is outerplanar by criterion $2$, we can embed the cone over $G$ in the plane. When we delete the top of the cone, the connected component of the point corresponding to the deleted vertex is the interior of a face that contains all vertices of $G$. So there is a face containing all vertices of $G$, which shows that criterion $2$ implies criterion $1$.

\begin{defn}
Consider a vertex $v$ of the $2$-complex~$C$ and define the following graph. Its vertices are the edges of~$C$ incident to $v$ and two vertices of~$L(v)$ are connected by an edge if their corresponding edges in~$C$ lie on the same face. This graph is called the \emph{link graph} of $v$ and is denoted by~$L(v)$.
\end{defn}

\begin{defn} \label{l2c}
A $2$-complex whose link graphs are all $k$-connected simple graphs is called \emph{locally $k$-connected}.
\end{defn}

\begin{defn}
Given a 2-complex~$C$ without loops, the \emph{(2-dimensional) cone} over~$C$ is the following 
2-complex. It is obtained from~$C$ by adding a single vertex (referred to as the \emph{top} of the cone), 
one edge for every vertex of~$C$ from that vertex to the top, and one triangular face for every edge 
$e$ of~$C$ whose endvertices are the endvertices of $e$ and the top. 

We denote the 2-dimensional cone over a $2$-complex~$C$ by $\widehat C$.  
\end{defn}
\begin{obs}
The link graph at the top of a cone of a $2$-complex~$C$ is equal to the $1$-skeleton of~$C$.
\qed
\end{obs}
\begin{obs}
2-dimensional cones are always simply connected.\qed
\end{obs}
\begin{defn}
The \emph{geometric realisation} of a $2$-complex $C=(V,E,F)$ is the topological space obtained by gluing discs to the geometric realisation of the graph $G=(V,E)$ along the face boundaries.
\end{defn}  
\begin{defn}
A (topological) embedding of a simplicial complex~$C$ into a topological space $X$ is an injective continuous map from (the
geometric realisation of)~$C$ into $X$. We say that a $2$-complex is \emph{embeddable} in $\mathbb{R}^3$ if its geometric realisation is embeddable in $\mathbb{R}^3$ as a topological space.
\end{defn}
\begin{rem}
If a $2$-complex is embeddable in $\mathbb{R}^3$ we will say as a shorthand that the $2$-complex is \emph{embeddable}.
\end{rem}
Now we are ready to give the definition of an outerspatial $2$-complex. We have the two definitions of outerplanarity \autoref{OP1} and \autoref{OP2}. We can make two different definitions for outerspatial $2$-complexes based on them.

\begin{defn} (Outerspatiality criterion 1)
A $2$-complex~$C$ is \emph{weakly outerspatial} if there is an embedding of~$C$ in $\mathbb R^3$ such that some chamber of this embedding is incident to all of the edges of the $2$-complex.
\end{defn}
\begin{defn}(Outerspatiality criterion 2)
A 2-complex is \emph{outerspatial} if its 2-dimensional cone embeds in 3-space. 
\end{defn}

It would be best if these two definitions were equivalent in the same way the two outerplanarity definitions are. However, this is not the case. It turns out that outerspatiality criterion $2$ is a stronger definition as shown by the lemma below.
\begin{lem}
If a $2$-complex~$C$ is outerspatial, then it is also weakly outerspatial.
\end{lem}
\begin{proof}
Consider an embedding of the cone over~$C$ in $\mathbb{R}^3$; assume that the cone is embedded in the outer face. Delete the top of the cone with all incident edges and faces. The chamber where the top was includes all faces incident with the top and thus has all edges of~$C$ in its boundary. Thus, this defines a weakly outerspatial embedding of~$C$.
\end{proof}

We showed that outerspatiality implies weak outerspatiality. To show that it is a strictly stronger definition we need an example of $2$-complex which is outerspatial but not weakly outerspatial. This is the Bing house -- \autoref{bing}.

For this paper we have chosen the second definition of outerspatial (\autoref{OP2}) as it is more general and yields more interesting and relevant characterisations.

Below, we will need to use a notion of inside and outside of a sphere. The following theorem provides the definition that we want.
\begin{thm}\label{jordan}
(Jordan–Brouwer separation theorem \cite{JB}) Any compact, connected hypersurface $X$ in $\mathbb R^n$ will divide $\mathbb R^n$ into two connected regions; the `outside' $D_0$ and the `inside' $D_1$. Furthermore, $\bar{D}_1$ is itself a compact manifold with boundary $\partial \bar{D}_1=X$.
\end{thm}
\begin{defn}
The \emph{interior of a cycle} in a plane embedding of a graph is the inside of its image as defined in \autoref{jordan}.
\end{defn}

\begin{obs} \label{link_planar}
All the link graphs of an embeddable $2$-complex are planar; for example, see \cite[Section 3]{JC2}.\qed
\end{obs}

\begin{defn}
Consider two cycles embedded in the plane. We say that they \emph{intersect internally} if their interiors have a proper non-empty intersection.
\end{defn}
\begin{defn} \label{nested_emb}
Consider a planar graph $G$ with a set of cycles $\Ccal$. We say that a plane embedding of $G$ is \emph{nested} if no two cycles in $\Ccal$ intersect internally.

\end{defn}

\begin{defn} \label{outerspherical}
Take a graph embedded in a 2-sphere $S$ which is in turn embedded in the Euclidean space $\mathbb{R}^3$. Glue discs to cycles of this graph inside the sphere so that they do not intersect each other in interior points, and they intersect the sphere $S$ precisely in the gluing cycles. Call a topological space that can be obtained in this way \emph{outerspherical}.

\end{defn}

\begin{defn} \label{candidate_outerspatial}
We will call a $2$-complex~$C$ \emph{outerspherical}, if it has an embedding in $\mathbb{R}^3$ that is an outerspherical topological space, where the $1$-skeleton of~$C$ is mapped to the graph in the sphere and the faces of~$C$ are mapped to the discs glued to the graph as described in \autoref{outerspherical}.
\end{defn}
\begin{defn}
Consider a graph $G$ with a set of vertices $V$. We will call a cyclic orientation of the edges incident to a vertex $v\in V$ a \emph{rotator} at $v$.
\end{defn}
Before we start with the next lemma, we need a definition which will help us differentiate between the faces of a graph and the faces of a $2$-complex.
\begin{defn}
    Consider an embedding of a planar graph $G$ on the sphere $\mathbb S^2$. We will call the connected components of $\mathbb S^2 \backslash G$ the \emph{facets} of this embedding.
\end{defn}
\begin{lem}\label{plane_to_outerspatial}
Let~$C$ be a $2$-complex. Then the following are equivalent
\begin{enumerate} [(1)]
\item $C$ is outerspherical
\item $C$ is outerspatial
\item The $1$-skeleton of~$C$ together with the set of face boundaries of~$C$ has a nested plane embedding.
\end{enumerate}
\end{lem}

\begin{proof}

For the $(1)\implies(2)$ implication, consider an outerspherical $2$-complex~$C$. There exists an embedding of it in $\mathbb R^3$ such that the unit sphere intersects~$C$ in its 1-skeleton and everything else of~$C$ is embedded in the interior of the unit ball. Let $G$ be the plane embedding of the 1-skeleton of~$C$ of this embedding in the unit ball.
The cone over $\mathbb{S}^2$ is the full 3-dimensional unit ball. Assume that $G$ is drawn onto $\mathbb{S}^2$. This way we obtain an embedding of the cone over $G$ in the full unit ball that takes the embedding $G$ into $\mathbb{S}^2$ at the boundary.
Glue these two full unit balls by gluing at $G$ to obtain an embedding of the cone over~$C$ in $\mathbb R^3$. This shows that~$C$ is outerspatial.

For the $(2)\implies(3)$ implication, consider a $2$-complex~$C$ and suppose that it is outerspatial. In other words, its cone $\widehat{C}$ is embeddable. Consider an embedding of $\widehat{C}$ and let the top of this cone be $t$. The embedding of $\widehat{C}$ induces an embedding of the link graph at the top in the 2-sphere. As the link graph at the top is the 1-skeleton of~$C$, the embedding of $\widehat{C}$ induces an embedding of the $1$-skeleton of~$C$ in the plane. Call that embedding $\iota$. 
Denote the 1-skeleton of~$C$ by $G$ and the set of face boundaries of~$C$ by $F$. 

\begin{sublem}\label{sublem2}
The embedding $\iota$ of $G$ is a nested plane embedding for the set $F$ of cycles. 
\end{sublem}

\begin{proof}
Suppose not for a contradiction.
As the embedding $\iota$ is plane by construction, there are two cycles in the set $F$ that are not nested; that is, they intersect internally.
So by \autoref{jordan}, there is a vertex $v$ of the link graph $L(t)$ such that the cycles intersect internally at $v$. That is, there are edges $(e_i|i\in\Zbb_4)$ of $L(t)$ incident with $v$ that appear in that order at the rotator at $v$ such that $e_1$ and $e_3$ are in one of the cycles of $F$ and $e_2$ and $e_4$ are in the other of these cycles. 

Now we find a vertex $w$ of~$C$ such that the embedding of the link graph $L(w)$ induced by the embedding of $\widehat C$ is not planar; this will be the desired contradiction. 
Let $w$ be the endvertex of $v$, considered as an edge of the cone $\widehat{C}$, aside from the top $t$. 
Consider the link graph $L(w)$. Now $v$ is an edge of $L(w)$, and also the edges $e_i$ are edges of $L(w)$ incident with $v$; and the rotator at $v$ is the same as in $L(t)$ (up to reversing).
Here, however, as the edges $e_1$ and $e_3$ are in a common face of~$C$, their endvertices in $L(w)$ aside from $v$ are joined by an edge, call it $x_1$. Similarly, the endvertices of the edges $e_2$ and $e_4$ aside from $v$ are joined by an edge; call it $x_2$.
The subgraph of $L(w)$ with the six edges $(e_i|i\in \Zbb_4)$ and $(x_i|i=1,2)$ is not planar with the specified rotator at $v$, see \autoref{nonplanar_eg}. This is the desired contradiction. 
\end{proof}

\begin{figure}[ht] 
\centering
\includegraphics[width=\textwidth]{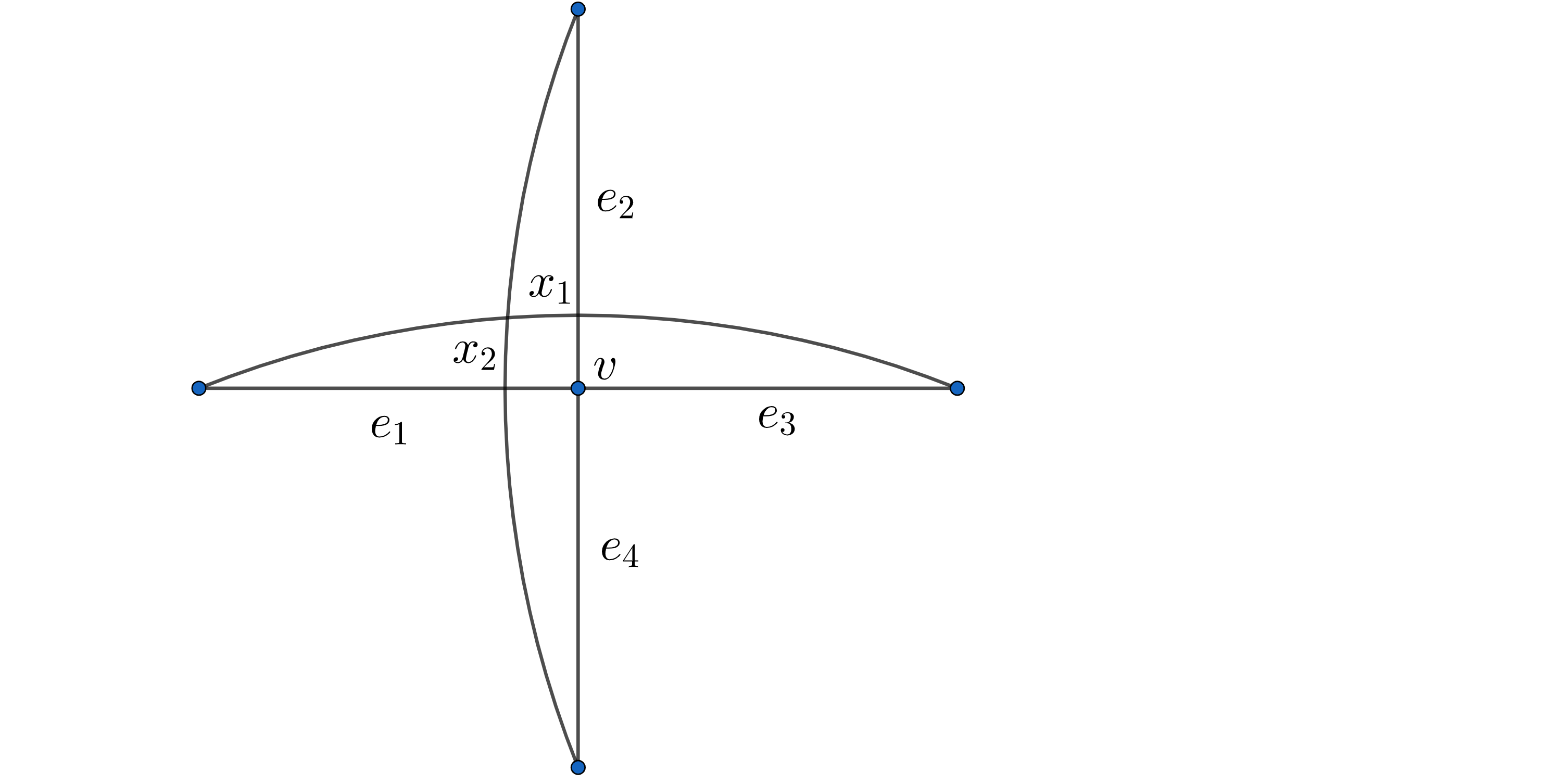}
\caption{A subgraph of the link graph $L(w)$ that is not planar for the induced rotator at the vertex $v$.}
\label{nonplanar_eg}
\end{figure}

We started with the assumption that~$C$ is outerspatial and we obtained the result from \autoref{sublem2}. This completes the $(2)\implies(3)$ implication. 

For the $(3) \implies (1)$ implication suppose that we have a complex~$C$ such
that its 1-skeleton $G$ together with the set $F$ of face boundaries has a nested plane
embedding.  We prove by induction on the number of elements of $F$ that do not bound a facet of the nested plane embedding that~$C$ has an outerspherical embedding. If the number is zero, we get that~$C$ is an embedding of $G$ on the sphere, which covers the base case.
Now, let $f\in F$ that does not bound a facet be given. By \autoref{jordan}, $f$ divides the plane embedding of $G$ into two subgraphs $G_1$ and $G_2$ intersecting at the cycle $f$. Let $F_i$ be the set of faces in $F$ attaching at $G_i$ for $i=1,2$. As $f$ does not bound a facet, each $G_i$ is strictly smaller than $G$. Hence by induction, each 2-complex $C_i$ obtained from $G_i$ by adding the faces from $F_i$ has an outerspherical embedding, which includes the face $f$. Now glue these two embeddings at the face $f$ to obtain an outerspherical embedding of $G$.

We proved that implications $(1)\implies(2)\implies(3)\implies(1)$, therefore all these three statements are equivalent as claimed.

\end{proof}

The equivalence of a $2$-complex being outerspatial and having a nested plane embedding of the $1$-skeleton is an interesting connection between graphs and $2$-complexes which will also be important for the proof of the main result.

The equivalence of a $2$-complex being outerspatial and being outerspherical is a good geometric characterisation of outerspatial $2$-complexes, that does not require further assumptions, which is also going to be useful to this paper. However, we want to find a more concrete characterisation, particularly with forbidden minors, which we will do in \autoref{section_core_lemma}.

Next, we prepare to give the details for \autoref{rem99} from the introduction.
\begin{lem}\label{is_nested} (Folklore)
Let $G$ be a planar graph without parallel edges and $\Ccal$ be any set of triangles of $G$. Then 
any plane embedding of $G$ is nested. 
\end{lem}
\begin{proof}
Since any two cycles which are triangles cannot intersect internally, the result immediately follows from \autoref{nested_emb}.
\end{proof}

\begin{prop} \label{main_naive_lemma}
A $2$-dimensional simplicial complex is outerspatial if and only if its $1$-skeleton is planar.
\end{prop}
\begin{proof}
For the `only if' implication consider an outerspatial $2$-dimensional simplicial complex~$C$. Since~$C$ is outerspatial, its cone $\widehat{C}$ is embeddable. The link graph $L(t)$ at the top $t$ of the cone is equal to the $1$-skeleton of~$C$. By \autoref{link_planar} and the fact that $\widehat{C}$ is embeddable, the graph $L(t)$ is planar. Since $L(t)$ is planar and is equal to the $1$-skeleton of~$C$, it follows that the $1$-skeleton of~$C$ is planar.

For the `if' implication, suppose that the $1$-skeleton of~$C$ is planar. Since~$C$ is a simplicial complex, its $1$-skeleton has no parallel edges, and all face boundaries are triangles. Thus, by \autoref{is_nested} the 1-skeleton has a nested plane embedding. So, the $2$-dimensional simplicial complex is outerspatial by  \autoref{plane_to_outerspatial}.
\end{proof}

\begin{rem}
For $2$-dimensional simplicial complexes, the problem of outerspatiality is simple, however, this changes when one allows larger faces. That is because \autoref{main_naive_lemma} follows easily from \autoref{is_nested} and \autoref{plane_to_outerspatial} but for the $2$-complex analogue of \autoref{main_naive_lemma} we do not have an analogue of \autoref{is_nested} that we can use so we cannot apply \autoref{plane_to_outerspatial} as easily. Checking for outerspatiality becomes significantly more complex when moving from the class of 2-dimensional simplicial complexes to general 2-complexes.  This is because the associated nested plane embedding problem becomes nontrivial.  This motivates why we work with $2$-complexes in the rest of the paper.
\end{rem}

\section{Core Lemma} \label{section_core_lemma}
In this section we are going to introduce the Core Lemma, which is the main ingredient in proving \autoref{mainintro}. We will prove it in the next section.

\begin{defn}
In this paper we will call a compact connected $2$-dimensional topological manifold without boundary a \emph{surface}.
\end{defn}
\begin{defn}
A $2$-complex is \emph{aspherical} if its geometric realisation is a surface which is not homeomorphic to the $2$-sphere. 
\end{defn}

\begin{lem}\label{a-sphere-cone}
Aspherical $2$-complexes are not outerspatial.
\end{lem}
\begin{proof}
Consider an aspherical $2$-complex~$C$ with a $1$-skeleton $G$. Then the link graph at the top of the cone is equal to $G$. Since $G$ is a triangulation of a surface of positive genus, it cannot be planar because it does not have the required Euler characteristic. But the link graphs of an embeddable $2$-complexes are planar by \autoref{link_planar}. Hence, $\widehat{C}$ does not embed in $\mathbb R^3$ which means that~$C$ is not outerspatial.
\end{proof}
\begin{lem}\label{contract_edge}
Contraction of non-loop edges preserves being outerspatial. 
\end{lem}

\begin{proof}
Clearly being outerspherical is preserved by contracting non-loop edges. So, this follows from \autoref{plane_to_outerspatial}.
\end{proof}

\begin{lem}\label{link_outerplanar}
All link graphs of an outerspatial $2$-complex are outerplanar. 
\end{lem}

\begin{proof}
Consider an outerspatial $2$-complex~$C$ and a link graph~$L(v)$ at an arbitrary vertex $v$ of~$C$. Let $\widehat{C}$ be the cone over~$C$ with a top $t$ and let $\widehat{L(v)}$ be the link graph at $v$ as a vertex of $\widehat{C}$.
\begin{sublem} \label{cone_link}
$\widehat{L(v)}$ is the $1$-dimensional cone over~$L(v)$.
\end{sublem}
\begin{proof}
Recall that the link graph~$L(v)$ has a vertex for each edge of~$C$ and two vertices in~$L(v)$ are connected by an edge if they share a face in~$C$. To build the cone $\widehat{C}$ from the $2$-complex~$C$, we add one new edge incident to $v$, namely $tv$, so~$L(v)$ has one new vertex. For each edge $uv$ incident to $v$ in~$C$ we add one face $tuv$, incident to the edge $tv$. Therefore, for any vertex $uv$ of~$L(v)$ we add one edge between $tu$ and $uv$. Thus, we showed that to obtain $\widehat{L(v)}$ from~$L(v)$ we add one new vertex and connect it to all old vertices by an edge. So $\widehat{L(v)}$ is the cone over~$L(v)$ as claimed.
\end{proof}
Since~$C$ is outerspatial, $\widehat{C}$ is embeddable in $3$-space by definition. Therefore, by \autoref{link_planar}, we know that $\widehat{L(v)}$ is planar. By \autoref{cone_link}, we know that $\widehat{L(v)}$ is the cone over~$L(v)$. In other words, the cone over~$L(v)$ is planar. This is precisely the definition of an outerplanar graph. Since $v$ was arbitrary, we showed that all link graphs of~$C$ are outerplanar, as desired.
\end{proof}

\begin{lem} \label{core_lemma}
(Core Lemma) A simple locally $2$-connected $2$-complex~$C$ is outerspatial if and only if
it does not contain an aspherical 2-complex as a subcomplex, and it does not contain a path $P$ 
such that the link graph at the vertex $P$ of $C/P$ is not outerplanar. 
\end{lem}

\section{Main techniques} \label{section_techniques}
In this section we are going to prove two lemmas that are needed for one of the implications of \autoref{core_lemma}. At the end of the section we will show the proof of \autoref{core_lemma} given these two lemmas and consequently show the proof of \autoref{mainintro} given \autoref{core_lemma} . We will start the section with a definition that will be used in the context of both lemmas.
\begin{defn} 
We are going to call a $2$-connected simple outerplanar graph a \emph{bi-outerplanar} graph.
\end{defn}
The class of bi-outerplanar graphs will be important for the proofs of these lemmas. Thus, before we get to proving them, we are going to need to explore bi-outerplanar graphs through definitions and a few small lemmas.

\begin{thm}(\cite{Outerminor})  \label{link_minors}
The set of excluded minors for the class of outerplanar graphs consists of the graphs $K_4$ and $K_{2,3}$.
\end{thm} 
\begin{obs} \label{equiv}
The link graphs of an outerspatial $2$-complex do not have $K_4$ or $K_{2,3}$ minors.
\end{obs}
\begin{proof}
The link graphs of an outerspatial $2$-complex are all outerplanar by \autoref{link_outerplanar}. The class of outerplanar graphs is characterised by its forbidden minors $K_4$ and $K_{2,3}$. From this, the conclusion follows.
\end{proof}
\begin{lem} \label{polygonal}(Folklore)
Every bi-outerplanar graph $G$ has a unique Hamiltonian cycle which bounds its outer face.
\end{lem}

Using the notions from the previous lemma we have the following definitions.
\begin{defn}
In a bi-outerplanar graph, we pick a cycle as in \autoref{polygonal} and refer to it as the \emph{boundary cycle}. We call an edge \emph{diagonal} if it connects two non-consecutive edges of the boundary cycle.
\end{defn}

\begin{defn}
A face of a $2$-complex~$C$ is called \emph{diagonal} if it is a diagonal edge in some of the link graphs of 
$C$. 
\end{defn}
\begin{defn}
A diagonal face of a $2$-complex is called \emph{perfectly diagonal} if it is diagonal in the link graphs of all 
of its endvertices.
\end{defn}

\begin{lem} \label{diagonal_are_perfect}
Consider a simple $2$-complex~$C$. If the link graphs of the complexes $C/P$ for paths $P$ of~$C$ (possibly trivial) are all $2$-connected outerplanar graphs, then all of its diagonal faces are perfectly diagonal.
\end{lem}

\begin{lem} \label{2-outerplanar_outerspatial}
Suppose that the simple $2$-complex~$C$ does not have an aspherical subcomplex. If the diagonal faces of~$C$ are all perfectly diagonal and the link graphs of the complexes $C/P$ for paths $P$ of~$C$ (possibly trivial) are all bi-outerplanar, then~$C$ is outerspatial.
\end{lem}

The next two subsections are devoted to proving these two lemmas.

\subsection{Proof of \autoref{diagonal_are_perfect}}

\begin{defn}
Let $H_1$ and $H_2$ be two graphs with a common vertex $v$ and a bijection $\iota$ between the edges incident with $v$ in $H_1$ and $H_2$. The \emph{vertex sum} of $H_1$ and $H_2$ over $v$ given $\iota$ is the graph obtained from the disjoint union of $H_1$ and $H_2$ by deleting $v$ in both $H_i$ and adding an edge between any pair $(v_1; v_2)$ of vertices $v_1 \in V (H_1)$ and $v_2 \in V (H_2)$ such that $v_1v$ and $v_2v$ are mapped to one another by $\iota$.

\end{defn} 

\begin{proof}[Proof of \autoref{diagonal_are_perfect}]
    If there are no diagonal faces, the claim is vacuously true. Suppose that there is a diagonal face $f$ bounded by the cycle with vertices $u,x,x_1,\dots,x_n$ and edges $ux,xx_1,x_1x_2,\dots,x_nu$, which is a chord in the link graph $L(u)$ at the vertex $u$. The edge $e_x=ux$ is a vertex in the link graphs $L(u)$ and $L(x)$ and is of the same degree in both. The degree $d$ of $e_x$ in $L(u)$ is at least three as it is the endvertex of a chord so the degree of $e_x$ in $L(x)$ is also at least three, hence it is also the endvertex of a chord. We claim that the face $f$ is a chord in $L(x)$. We will prove this using the following.
\begin{sublem}
If the face $f$ is not a chord in $L(x)$, then the link graph $L(e_x)$ at the vertex $e_x$ of the $2$-complex $C/e_x$ has a $K_{2,3}$ minor.
\end{sublem}
\begin{proof}
Suppose for a contradiction that the face $f$ is not a chord in $L(x)$. By \cite[Observation 3.1.]{JC1} the link graph $L(e_x)$ at the vertex $e_x$ in the $2$-complex $C/e_x$ is equal to the vertex sum of the graphs $L(u)$ and $L(x)$ in~$C$ at their common vertex $e_x$. Denote $H$ to be this vertex sum. We shall see how the edges incident to $e_x$ in $L(u)$ and $L(x)$ get identified to obtain $H$. Recall that $e_x$ is the endvertex of at least one chord in $L(x)$. Since the chord $f$ in $L(u)$ is a non-chord in $L(x)$, by the pigeonhole principle one of the chords in $L(x)$ incident to $e_x$ is a non-chord in $L(u)$, call that chord $g$. Let the other end of $f$ in $L(u)$ be $e_x'$ and the other end of $g$ in $L(x)$ be $e_x''$. Since $L(x)$ is bi-outerplanar, there are two non-chord edges incident to $e_x$, one of which is $g$, let the other one be $k$. Now we have the following.

There is a path of length at least two from $e_x'$ to $e_x$ through $g$ in $L(u)$ and a path of length at least one from $e_x$ to $e_x''$ through $g$ in $L(u)$ and the same for~$k$. There is a path of length at least one from $e_x'$ to $e_x$ through $f$ in $L(u)$ and a path of length at least two from $e_x$ to $e_x''$ through $f$ in $L(x)$. All the paths mentioned above are pairwise internally vertex disjoint. Therefore, in the vertex sum $H$ there are three internally vertex disjoint paths between the same pair of endvertices containing $f$, $g$ and $k$ respectively each of length at least two. This yields a subdivision of $K_{2,3}$. So $L(e_x)=H$ has a subgraph that is a subdivision of $K_{2,3}$. This means that $L(e_x)$ has a $K_{2,3}$ minor as claimed. 
\end{proof}
\begin{figure}[ht] 
\centering
\includegraphics[width=\textwidth]{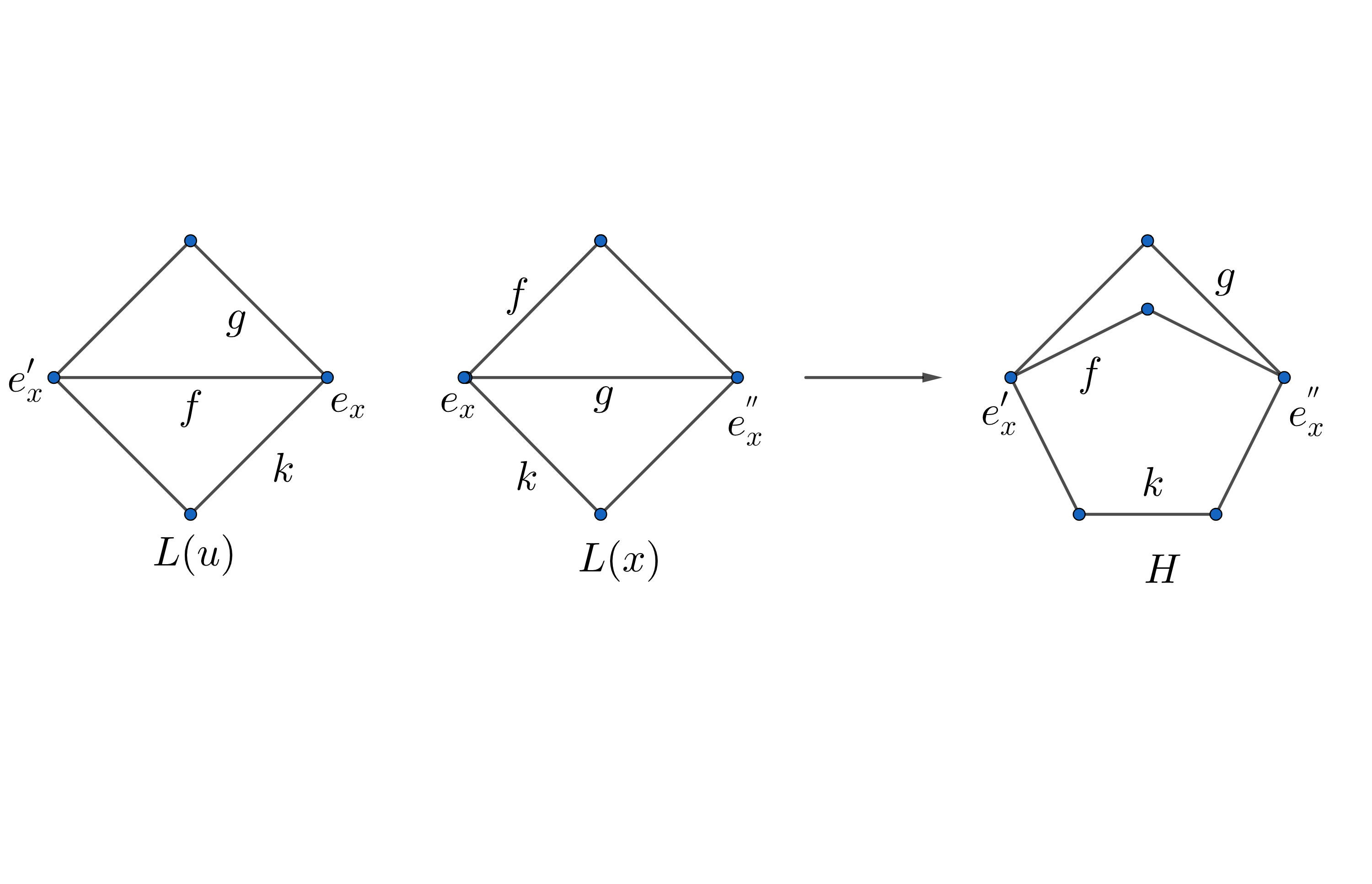}
\caption{$H$ is the vertex sum of $L(u)$ and $L(x)$ at the vertex $e_x$ and has a $K_{2,3}$ minor.}
\label{pic}
\end{figure} 
If the graph $L(e_x)$ has a $K_{2,3}$ minor, then $C/e_x$ has a link graph that is not outerplanar. Using \autoref{equiv}, this means that $C/e_x$ is not outerspatial. Now, from \autoref{contract_edge}, we conclude that~$C$ is also not outerspatial, which is a contradcition with our assumptions. This yields that $f$ is a chord in $L(x)$. Similarly, looking at the link graphs $L(x)$ and $L(x_1)$, we obtain that $f$ is a chord in $L(x_1)$. Repeating this argument inductively, we obtain that $f$ is a chord in each $L(u)$, $L(x)$ and $L(x_i)$, $1\leq i \leq n$. Therefore, $f$ is a chord in the link 
graph of all of its endvertices. Since $f$ was arbitrary, we proved that every diagonal face is 
perfectly diagonal as claimed. 
\end{proof}

\subsection{Proof of \autoref{2-outerplanar_outerspatial}}

\begin{proof}[Proof of \autoref{2-outerplanar_outerspatial}]
If~$C$ has no diagonal faces, then all link graphs are cycles, since they are all bi-outerplanar. This means that that the geometric realisation of~$C$ is homeomorphic to a surface. By the assumption of the lemma, in such a case the geometric realisation of~$C$ must be homeomorphic to a sphere. A sphere is outerspatial and homeomorphism preserves outerspatiality, so we are done in this case.

Now, suppose that~$C$ has a diagonal face and consider one such face $f$. The link graph at each endvertex is bi-outerplanar and $f$ is a chord in each of these link graphs. Therefore, if we remove $f$ from the $2$-complex~$C$, we only remove chords from link graphs of~$C$ and they stay bi-outerplanar. Thus, we can remove the diagonal faces of~$C$ one by one to arrive at a $2$-complex $D$ whose link graphs are all cycles. As seen above, the geometric realisation of $D$ must be homeomorphic to a sphere. Because of that and the fact that the $1$-skeleton of $D$ is naturally embedded in $D$, we can view the $1$-skeleton of $D$ as a plane graph.
\begin{sublem} \label{boundaries_are_nested}
The $1$-skeleton of $D$ together with the boundaries of the removed diagonal faces form a nested plane embedding.
\end{sublem}
\begin{proof}
Suppose for a contradiction, that there are two diagonal faces $f_1$ and $f_2$ with face boundaries $c_1$ and $c_2$ which are not nested. By \autoref{jordan}, $c_1$ divides $D$ into two connected components. Since $c_1$ and $c_2$ are not nested, there are edges of $c_2$ in both connected components of $D\backslash c_1$. Therefore, there exists a subpath of $c_2$ that starts in one of the connected components of $D\backslash c_1$ and ends in the other. Choose a minimal such path and call it $p$. Contract the subpath $p'$ of $p$ that consists of $p$ with the first and last edge removed. Let the complex obtained from this contraction be $D'=D/p'$. The path $p'$ is a subpath of $c_1$ by minimality of $p$, and is obviously a subpath of $c_2$. Let $c_1'=c_1/p'$ and $c_2'=c_2/p'$. 

In the $2$-complex $D'$ we find two consecutive edges of $c_2'$, each in a different component of $D'\backslash c_1'$. Call the two edges $a_2$ and $b_2$ and notice that $p'$ is the vertex that these edges share. Let the two edges of $c_1'$ incident with $p'$ be $a_1$ and $b_1$. By assumption, the link graph $L(p')$ is bi-outerplanar and thus it is a cycle~$C$ together with a set of edges between the vertices of~$C$ by \autoref{polygonal}. In $L(p')$, the vertices $a_1$ and $b_1$ are connected by the edge $f_1$ and the vertices $a_2$ and $b_2$ are connected by the edge $f_2$. Since $f_1$ and $f_2$ are perfectly diagonal, they are chords in $L(p')$. The above shows that $L(p')$ contains a cycle together with two non-parallel chords as a subgraph. The latter has a $K_4$ minor and so $L(p')$ also has a $K_4$ minor. This means that the $2$-complex $D/p'$ for the path $p'$ has a non-outerplanar link graph at the vertex $p'$. The link graph $L(p')$ in $D/p'$ is a subgraph of the link graph $L(p')$ in $C/p'$. Therefore, $C/p'$ also has a non-outerplanar link graph at the vertex $p'$. This is a contradiction with the assumptions of the lemma which tells us that the boundaries of any two removed diagonal faces are nested, from which the sublemma follows. 
\end{proof} 
As Let $\mathcal{C}_1$ denote the set of boundaries of the faces of $D$. Let $\mathcal{C}_2$ denote the set of boundaries of the removed diagonal faces. Then the set $\mathcal{C}$ of boundaries of the faces of~$C$ satisfies $\mathcal{C}=\mathcal{C}_1\cup \mathcal{C}_2$. \autoref{boundaries_are_nested} gives us that the $1$-skeleton of $D$ together with $\mathcal{C}_2$ form a nested plane embedding. The $2$-complexes~$C$ and $D$ have the same $1$-skeleton and the elements of $\mathcal{C}_1$ are nested with any other element in $\mathcal{C}$, because they are boundaries of faces on the sphere $D$. These two facts and \autoref{boundaries_are_nested} together give us that the $1$-skeleton of~$C$, together with the boundaries of the faces of~$C$ form a nested plane embedding. From this it follows that~$C$ is outerspatial by \autoref{plane_to_outerspatial} as claimed, which finishes the proof.
\end{proof}
\subsection{Proof of \autoref{core_lemma} and \autoref{mainintro}.}
\begin{proof}[Proof of \autoref{core_lemma}:]
For the `only if' direction, assume that the 2-complex~$C$ has a subcomplex that is an aspherical 2-complex. As being 
outerspatial is closed under deletion of faces, it follows from \autoref{a-sphere-cone} that the 
2-complex~$C$ cannot be outerspatial.

Next assume that the 2-complex~$C$ contains a path $P$ such that the link graph at the vertex $P$ of $C/P$ is not outerplanar.  Since $L(P)$ is not outerplanar, it follows from \autoref{link_outerplanar} that $C/P$ is not outerspatial. Then by \autoref{contract_edge} it follows that~$C$ is also not outerspatial.

We proved that, if a $2$-complex~$C$ contains an aspherical subcomplex or a path $P$ such that the link graph at the vertex $P$ of $C/P$ is not outerplanar, then~$C$ is not outerspatial. This proves the `only if' direction as required.

For the `if' implication, consider a $2$-complex~$C$ which does not contain an aspherical $2$-complex as a subcomplex and does not contain a path $P$ such that the link graph at the vertex $P$ of $C/P$ is not outerplanar. Then we can first apply \autoref{diagonal_are_perfect} to obtain that all of its diagonal faces are perfectly diagonal. Next, since the result of \autoref{diagonal_are_perfect} completes the assumptions of \autoref{2-outerplanar_outerspatial}, we can apply the latter to obtain the final result.
\end{proof}

\begin{proof}[Proof of \autoref{mainintro}:]
\autoref{mainintro} states that a simple locally $2$-connected $2$-complex~$C$ is outerspatial if and only if it does not contain a surface of positive genus as a subcomplex or $2$-complex with non-outerplanar link graph as a space minor. \autoref{core_lemma} states that a simple locally $2$-connected $2$-complex~$C$ is outerspatial if and only if it does not contain an aspherical  $2$-complex as a subcomplex and it  does not contain a path $P$ such that the link graph at the vertex $P$ of $C/P$ is not outerplanar.

Looking at both of these statements we can see that they are of the form `a locally $2$-connected $2$-complex~$C$ is outerspatial if and only if it does not contain some forbidden structures'. So, to prove the implication we need to prove that the forbidden structures in \autoref{mainintro} are a subset of the forbidden structures in \autoref{core_lemma}.

Aspherical $2$-complexes are homeomorphic to surfaces of positive genus. Since contracting a path is a space minor operation, if a $2$-complex~$C$ contains a path $P$ such that $C/P$ has a non-outerplanar link graph, then it has a space minor with a non-outerplanar link graph.

The previous paragraph proves that the set of forbidden structures in \autoref{core_lemma} contains the set of forbidden structures in \autoref{mainintro} and thus \autoref{core_lemma} implies \autoref{mainintro} as claimed.
\end{proof}

\section{Further remarks on locally 2-connected outerspatial 2-complexes}\label{sec5}
We start this section by recalling the definition of an outerspherical topological space.
\begin{defn}
Take a graph embedded in a 2-sphere $S$ embedded as a 2-dimensional unit sphere in the Euclidean space $\mathbb{R}^3$. Glue discs to cycles of this graph inside the sphere so that they do not intersect each other in interior points, and they intersect the sphere $S$ precisely in the gluing cycles. Call a topological space that can be obtained in this way \emph{outerspherical}.
\end{defn}
Now we elaborate a bit further on this definition.
\begin{defn}
We call the discs glued to faces of the $1$-skeleton of the sphere \emph{outer} discs and all other discs we call \emph{inner} discs.
\end{defn}
\begin{defn}
The \emph{closure} of an outerspherical topological space $T$ is obtained by adding all missing outer discs.
\end{defn}

\begin{defn}
We will call an outerspherical topological space \emph{maximal} if it is equal to its closure.
\end{defn}
\begin{rem} \label{top_comp}
An outerspherical topological space induces a unique outerspatial $2$-complex by taking the $1$-skeleton of the $2$-complex to be the graph of the topological space and the faces of the $2$-complex to be the discs glued to cycles of this graph. Let a $2$-complex induced in such a way by an outerspherical topological space $T$ be denoted by $C(T)$.
\end{rem}
\begin{lem} \label{max_loc}
If $T$ is maximal, then $C(T)$ is locally $2$-connected.
\end{lem}
\begin{proof}
A $2$-complex~$C$, which can be embedded as a maximal outerspherical topological space is a subdivision of a sphere together with some additional faces. Since the link graphs of a subdivided sphere are cycles, which are $2$-connected, it follows that~$C$ is locally $2$-connected.
\end{proof} 
\begin{lem} \label{step1}
A simple outerspatial $2$-complex~$C$ is locally $2$-connected if and only if it has an embedding that is a maximal outerspherical topological space.
\end{lem}
\begin{proof}
For the `if' direction, let $T_C$ be an embedding of~$C$ that is a maximal outerspherical topological space. The result follows from \autoref{max_loc} and the fact that $C=C(T_C)$.

For the `only if' direction, consider a $2$-complex~$C$ which has an outerspherical embedding $T$ and note that $C=C(T)$. Let the closure of $T$ be denoted by $\overline{T}$. Suppose that $T$ is not maximal. This means that $\overline{T}\backslash T$ contains an outer disc, let one such disc be $f$. We have that $f$ is a face in $C(\overline{T})$. Let $v$ be one of the vertices of $f$ and let~$L(v)$ be the link graph of $v$ with respect to $C(\overline{T})$, we know by \autoref{link_outerplanar} that~$L(v)$ is outerplanar. A rotator at $v$ induces a Hamiltonian cycle $H$ in~$L(v)$ which is unique and bounds the outer face of~$L(v)$ by \autoref{polygonal}. In the graph~$L(v)$, let $f=ab$ and suppose that there are two paths $P_1$ and $P_2$ between $a$ and $b$ different from $ab$. The vertices of both of these paths are all vertices of $H$ due to the fact that $H$ is Hamiltonian. Since $H$ bounds the outer face of~$L(v)$, we have that all edges of $P_1$ are on the inside of $H$, thus the path $P_2$ lies on the inside of the face bounded by the cycle $P_1\cup f$, which leads to $P_1$ and $P_2$ intersecting internally. Therefore, the connectivity between $a$ and $b$ in the graph $L(v)\backslash ab$ is $1$ and so $L(v)\backslash ab$ is not $2$-connected. Hence, $C(\overline{T}\backslash f)$ is not locally $2$-connected and consequently $C(T)$ is also not locally $2$-connected. Since $C=C(T)$, it follows that~$C$ is not locally $2$-connected. With this, we proved the contrapositive of the `only if' statement and so we are done.
\end{proof}
\begin{defn}
The \emph{dual graph} of an outerspherical topological space is constructed in the following way. We have a vertex for each chamber apart from the outer chamber and two vertices are connected by an edge if their respective chambers share a disc.
\end{defn}
\begin{lem} \label{step2}
The dual graph of a maximal outerspherical topological space $T$ is a tree.
\end{lem}
\begin{proof}
We prove this by induction on the number of inner discs of $T$. When the number of inner discs is zero, we have a sphere, so the dual graph is a vertex which is a tree and thus the base case is true. Consider a maximal outerspherical topological space $T$ and suppose that all maximal outerspherical topological spaces with less inner discs have dual graphs that are trees.

Let $G$ be the dual graph of $T$ and suppose that we remove some inner disc $e$. Then the dual graph of $T-e$ is $G/e$. By the inductive hypothesis $G/e$ is a tree. Therefore, $G$ is also a tree, which completes the inductive step and thus completes the proof.
\end{proof}
\begin{prop} \label{locrem1}
Let~$C$ be a locally $2$-connected simple outerspatial $2$-complex. Then the dual graph of an embedding of~$C$ is a tree.
\end{prop}
\begin{proof}
Firstly note that the term dual graph of an embedding of~$C$ is well-defined by \autoref{plane_to_outerspatial}. From here, the result follows from \autoref{step1} and \autoref{step2}.
\end{proof}
\begin{cor} \label{uniq}
A simple outerspatial locally $2$-connected $2$-complex can be constructed by starting from a sphere and then gluing a sequence of spheres one by one at an already existing face.
\end{cor}

\begin{prop}
Every locally $2$-connected simple $2$-complex has a unique embedding up to combinatorial equivalence.
\end{prop}
\begin{proof}
By \autoref{uniq}, such a $2$-complex can be built by gluing a sequence of spheres at some faces. If we remove all the gluing faces we obtain a subdivided sphere. There is a unique way to embed the sphere and then we can embed back the gluing faces uniquely on the $1$-skeleton of this sphere. This shows that there exists a unique embedding of~$C$ (up to combinatorial equivalence).
\end{proof}

\begin{prop}
Every $n$-vertex locally $2$-connected simple outerspatial $2$-complex has at most $3n-6$ edges and at most $3n-8$ faces.
\end{prop}
\begin{proof}
Consider a locally $2$-connected simple outerspatial $2$-complex with $n$ vertices. That it has at most $3n-6$ edges follows from the fact that its $1$-skeleton is planar and from Euler's formula.

We will prove by induction on the number of spheres glued in \autoref{uniq} that there are at most $3n-8$ faces. 

For one sphere there are $n$ vertices and $2n-4$ faces. Since $n\geq 4$, we have that $3n-8 \leq 2n-4$ so the base case is true. Suppose that we have a $2$-complex with $n$ vertices at most $3n-8$ faces and we glue a sphere  with $m$ vertices and $2n-4$ faces at some face. The new number of vertices is $n+m-3$ and the new number of faces is at most $3n-8+2m-4-1=3n+2m-13$. Since $m\geq 4$, we have that $3n+2m-13\leq 3(n+m-3)-8=3n+3m-17$. This completes the inductive step and thus we proved that an $n$-vertex $2$-complex has at most $3n-8$ faces.
\end{proof}

\begin{lem} \label{lc}
Consider a $2$-complex~$C$ that is locally $2$-connected. Then the cone over it is locally $3$-connected.
\end{lem} 
\begin{proof}
For any vertex in~$C$, its link graph in $\widehat{C}$ is the ($1$-skeleton of the) cone over its link graph in~$C$. We know that if $G$ is $2$-connected, then (the $1$-skeleton of) its cone is $3$-connected, so the link graphs at the vertices in~$C$ are $3$-connected. Suppose that the link graph $L(t)$ of the top of the cone $t$  has a $2$-separator $\{tu,tv\}$. Then, restricting to~$C$, consider the link graph $L(u)$. If $uv \in E(C)$, then $v$ is a cutvertex in $L(u)$. If not, then $L(u)$ is disconnected. In either case, we have a contradiction with~$C$ being locally $2$-connected. If $L(t)$ has a $1$-separator $tu$, then $L(u)$ is disconnected when restricted to~$C$. This is again contradiction with~$C$ being locally $2$-connected. If $L(t)$ is disconnected, then so is~$C$ which is contradiction to   ~$C$ being simply-connected. We proved that $L(t)$ has no $0$-, $1$- or $2$-separators. Thus, the link graph at the top is also $3$-connected. All link graphs of $\widehat{C}$ are $3$-connected, therefore the cone over~$C$ is locally $3$-connected as required.
\end{proof}

\begin{prop}
Let~$C$ be a locally $2$-connected $2$-complex with $F$ faces. There exists an algorithm that checks in time linear in $F$ whether~$C$ is outerspatial.
\end{prop}

\begin{proof}
Given the locally $2$-connected $2$-complex~$C$ we can construct its cone $\widehat{C}$ in a linear time. $\widehat{C}$ is locally $3$-connected by \autoref{lc}. The methods of \cite{JC1} give an algorithm that checks in linear time whether a locally $3$-connected $2$-complex is embeddable. Given that algorithm we can check whether $\widehat{C}$ is embeddable in linear time. Since~$C$ is outerspatial if and only if $\widehat{C}$ is embeddable, this gives a linear algorithm that checks whether~$C$ is outerspatial.
\end{proof}

\bibliographystyle{plain}
\bibliography{References}

\begin{thebibliography}{10}

\bibitem{simply_nested}
Patrizio Angelini, Giuseppe Di~Battista, Michael Kaufmann, Tamara Mchedlidze,
  Vincenzo Roselli, and Claudio Squarcella.
\newblock Small point sets for simply-nested planar graphs.
\newblock In {\em International Symposium on Graph Drawing}, pages 75--85.
  Springer, 2011.

\bibitem{lamasad}
Arash Asadi, Luke Postle, and Robin Thomas.
\newblock Sub-exponentially many 3-colorings of triangle-free planar graphs.
\newblock {\em Electronic Notes in Discrete Mathematics}, 34:81--87, 2009.

\bibitem{JC1}
Johannes Carmesin.
\newblock Embedding simply connected 2-complexes in 3-space--i. a
  kuratowski-type characterisation.
\newblock {\em arXiv preprint arXiv:1709.04642}, 2017.

\bibitem{JC2}
Johannes Carmesin.
\newblock Embedding simply connected 2-complexes in 3-space--ii. rotation
  systems.
\newblock {\em arXiv preprint arXiv:1709.04643}, 2017.

\bibitem{JC3}
Johannes Carmesin.
\newblock Embedding simply connected 2-complexes in 3-space--iii. constraint
  minors.
\newblock {\em arXiv preprint arXiv:1709.04645}, 2017.

\bibitem{JC4}
Johannes Carmesin.
\newblock Embedding simply connected 2-complexes in 3-space--iv. dual matroids.
\newblock {\em arXiv preprint arXiv:1709.04652}, 2017.

\bibitem{JC5}
Johannes Carmesin.
\newblock Embedding simply connected 2-complexes in 3-space--v. a refined
  kuratowski-type characterisation.
\newblock {\em arXiv preprint arXiv:1709.04659}, 2017.

\bibitem{Outerminor}
Gary Chartrand and Frank Harary.
\newblock Planar permutation graphs.
\newblock In {\em Annales de l'IHP Probabilit{\'e}s et statistiques}, volume~3,
  pages 433--438, 1967.

\bibitem{lamepp}
David Eppstein and Bruce Reed.
\newblock Finding maximal sets of laminar 3-separators in planar graphs in
  linear time.
\newblock In {\em Proceedings of the Thirtieth Annual ACM-SIAM Symposium on
  Discrete Algorithms}, pages 589--605. SIAM, 2019.

\bibitem{lamfio}
Samuel Fiorini, Nadia Hardy, Bruce Reed, and Adrian Vetta.
\newblock Approximate min-max relations for odd cycles in planar graphs.
\newblock In {\em International Conference on Integer Programming and
  Combinatorial Optimization}, pages 35--50. Springer, 2005.

\bibitem{induced_packing}
Fabrizio Frati and Maurizio Patrignani.
\newblock A note on minimum-area straight-line drawings of planar graphs.
\newblock In {\em International Symposium on Graph Drawing}, pages 339--344.
  Springer, 2007.

\bibitem{lamgoe}
Michel~X Goemans and David~P Williamson.
\newblock Primal-dual approximation algorithms for feedback problems in planar
  graphs.
\newblock {\em Combinatorica}, 18(1):37--59, 1998.

\bibitem{min_1}
Petr~A Golovach, Marcin Kami{\'n}ski, Dani{\"e}l Paulusma, and Dimitrios~M
  Thilikos.
\newblock Induced packing of odd cycles in a planar graph.
\newblock In {\em International Symposium on Algorithms and Computation}, pages
  514--523. Springer, 2009.

\bibitem{hatcher}
Allen Hatcher.
\newblock {\em Algebraic topology}.
\newblock Cambridge University Press, 2005.

\bibitem{min_2}
Marcus Krug and Dorothea Wagner.
\newblock Minimizing the area for planar straight-line grid drawings.
\newblock In {\em International Symposium on Graph Drawing}, pages 207--212.
  Springer, 2007.

\bibitem{JB}
Wolfgang Schmaltz.
\newblock The jordan-brouwer separation theorem, 2009.

\bibitem{nested_h}
Won-Min Song, Tiziana Di~Matteo, and Tomaso Aste.
\newblock Nested hierarchies in planar graphs.
\newblock {\em Discrete Applied Mathematics}, 159(17):2135--2146, 2011.

\end{thebibliography}

\end{document}